\documentclass[11pt,english]{article}

\usepackage[T1]{fontenc}
\usepackage[latin9]{inputenc}
\usepackage{geometry}
\usepackage{babel}
\usepackage{amsmath}
\usepackage{amsthm}
\usepackage{amssymb}
\usepackage{setspace}
\makeatletter

\date{}

\usepackage{appendix}

\usepackage[bookmarks=true,bookmarksnumbered=false,bookmarksopen=false,pdfborder={0 0 0},pdfborderstyle={},backref=false,colorlinks=false]{hyperref}

\usepackage[nameinlink,capitalise,noabbrev]{cleveref}
\crefname{appsec}{Appendix}{Appendices}
\creflabelformat{enumi}{#2textup{#1}#3}

\begingroup
    \makeatletter
    \@for\theoremstyle:=definition,remark,plain\do{%
        \expandafter\g@addto@macro\csname th@\theoremstyle\endcsname{%
            \addtolength\thm@preskip\parskip
            }%
        }
\endgroup

\theoremstyle{plain}
\newtheorem{thm}{Theorem}
\crefname{thm}{Theorem}{Theorems}
\theoremstyle{plain}
\newtheorem{lem}[thm]{Lemma}
\crefname{lem}{Lemma}{Lemmas}
\theoremstyle{plain}

\theoremstyle{plain}
\newtheorem*{claim*}{Claim}
\theoremstyle{definition}

\theoremstyle{plain}

\theoremstyle{plain}

\theoremstyle{definition}

\theoremstyle{definition}

\theoremstyle{plain}

\crefformat{equation}{#2(#1)#3}

\let\originalleft\left
\let\originalright\right
\renewcommand{\left}{\mathopen{}\mathclose\bgroup\originalleft}
\renewcommand{\right}{\aftergroup\egroup\originalright}
\usepackage{pgfplots}
\usetikzlibrary{pgfplots.groupplots}
\usepackage{verbatim}

\makeatletter
\renewcommand*{\UrlTildeSpecial}{%
  \do\~{%
    \mbox{%
      \fontfamily{ptm}\selectfont
      \textasciitilde
    }%
  }%
}%
\let\Url@force@Tilde\UrlTildeSpecial
\makeatother

\usepackage{enumitem}

\makeatother

\begin{document}

\title{
\texorpdfstring{\vspace{-0.5cm}}{}
Counting Hamilton cycles in sparse random directed graphs}

\author{Asaf Ferber\thanks{Department of Mathematics, MIT. Email: \href{ferbera@mit.edu} {\nolinkurl{ferbera@mit.edu}}.
Research is partially supported by NSF grant 6935855.}\and Matthew Kwan\thanks{Department of Mathematics, ETH Z\"urich. Email: \href{mailto:matthew.kwan@math.ethz.ch} {\nolinkurl{matthew.kwan@math.ethz.ch}}.}\and
Benny Sudakov\thanks{Department of Mathematics, ETH Z\"urich. Email: \href{mailto:benjamin.sudakov@math.ethz.ch} {\nolinkurl{benjamin.sudakov@math.ethz.ch}}.}}

\maketitle

\global\long\def\E{\mathbb{E}}
\global\long\def\GG{\mathbb{G}}
\global\long\def\DD{\mathbb{D}}
\global\long\def\floor#1{\left\lfloor #1\right\rfloor }
\global\long\def\ceil#1{\left\lceil #1\right\rceil }

\begin{abstract}
Let $\DD\left(n,p\right)$ be the random directed graph on $n$ vertices
where each of the $n\left(n-1\right)$ possible arcs is present independently
with probability $p$. A celebrated result of Frieze shows that if $p\ge\left(\log n+\omega\left(1\right)\right)/n$
then $\DD\left(n,p\right)$ typically has a directed Hamilton cycle,
and this is best possible. In this paper, we obtain a strengthening of this result, showing that under the same condition,
the number of directed Hamilton cycles in $\DD\left(n,p\right)$ is
typically $n!\left(p\left(1+o\left(1\right)\right)\right)^{n}$. We
also prove a hitting-time version of this statement, showing that
in the random directed graph process, as soon as every vertex has
in-/out-degrees at least $1$, there are typically $n!\left(\log n/n\left(1+o\left(1\right)\right)\right)^{n}$
directed Hamilton cycles.
\end{abstract}

\section{Introduction}

A \emph{Hamilton cycle} in a graph is a cycle
passing through every vertex of the graph exactly once. We can similarly define a Hamilton cycle in a directed graph, with the extra condition that the edges along the cycle must be cyclically oriented. We say a graph or digraph
is \emph{Hamiltonian} if it contains a Hamilton cycle. Hamiltonicity
is one of the most central notions in graph theory, and has been
intensively studied by numerous researchers in recent decades. For example, the problems of deciding whether
a graph or digraph has a Hamilton cycle were both featured in Karp's
original list \cite{Kar72} of 21 NP-complete problems, and are closely
related to the travelling salesman problem.

The study of Hamilton cycles in \emph{random} graphs and digraphs goes back about 60 years, to the seminal paper of Erd\H os and R\'enyi on random graphs \cite{ER60}. They asked to determine the approximate ``threshold'' value of $m$ above which a random $m$-edge undirected graph is typically Hamiltonian. This question was famously resolved by P\'osa \cite{Pos76} and Korshunov \cite{Kor76}, pioneering
the use of a ``rotation-extension'' technique. There were a number of further improvements by different
authors, leading to the sharp ``hitting time'' result of Bollob\'as
\cite{Bol84} and Ajtai, Koml\'os and Szemer\'edi \cite{AKS85}.
A subsequent natural question (connected to ``robustness'' of Hamiltonicity, see for example \cite[Section~2]{Sud17}) is to estimate the \emph{number} of Hamilton cycles
in a random graph at or above this threshold. For example, Cooper and Frieze \cite{CF89} proved that above the (hitting-time) threshold for Hamiltonicity, random graphs have at least $(\log n)^{(1-\varepsilon)n}$ Hamilton cycles for any fixed $\varepsilon>0$. More recently, Glebov and Krivelevich \cite{GK13} proved that if $m$ is above this Hamiltonicity threshold, then a random $m$-edge undirected graph in fact has $n!\left(m/\binom n 2(1+o(1))\right)^n$ Hamilton cycles. See also the work of Janson \cite{Jan94}
on the number of Hamilton cycles in denser random graphs.

In this paper, we are interested in corresponding questions in the
directed case. Such questions have generally turned out to be harder,
as there is no general tool comparable to P\'osa's rotation-extension
technique. We will give a more precise account of the existing work
in this area, so we take the opportunity to formally define the different basic
models of random digraphs. Let $\DD\left(n,p\right)$ be the random
digraph on the vertex set $[n]:=\left\{ 1,\dots,n\right\} $
where each of the $N:=n\left(n-1\right)$ possible directed edges
$\left(v,w\right)$ (with $v\ne w$) is present independently with
probability $p$. Let $\DD\left(n,m\right)$ be the random digraph
consisting of a uniformly random subset of exactly $m$ of the possible
edges. We also define the random digraph \emph{process}, as follows.
Let
\[
e_{1}=\left(v_{1},w_{1}\right),e_{2}=\left(v_{2},w_{2}\right),\dots,e_{N}=\left(v_{N},w_{N}\right)
\]
be a random ordering of the ordered pairs of distinct vertices, and
let $D_{m}=\left\{ e_{1},\dots,e_{m}\right\} $. All three of these
models are very closely related: for each $m$, the marginal distribution
of $D_{m}$ is precisely $\DD\left(n,m\right)$, and for $p=m/N$,
the models $\DD\left(n,m\right)$ and $\DD\left(n,p\right)$ are in
a certain sense ``asymptotically equivalent'' (see \cite[Corollary~1.16]{JLR00}).

One of the first important insights concerning Hamilton cycles in
random digraphs, due to McDiarmid \cite{McD80}, is that one can use
coupling arguments to compare certain probabilities between $\DD\left(n,p\right)$
and the random undirected graph $\GG\left(n,p\right)$ (see \cite{Fer} for more recent applications of McDiarmid's idea). In particular,
using the known optimal results for $\GG\left(n,p\right)$, McDiarmid's
work implies that if $p\ge\left(\log n+\log\log n+\omega\left(1\right)\right)/n$
then $\DD\left(n,p\right)$ a.a.s.\footnote{By ``asymptotically almost surely'', or ``a.a.s.'', we mean that
the probability of an event is $1-o\left(1\right)$. Here and for
the rest of the paper, asymptotics are as $n\to\infty$.} has a Hamilton cycle. However, this is not optimal. Frieze \cite{Fri88}
later designed an algorithm to show that
if $p\ge\left(\log n+\omega\left(1\right)\right)/n$ then a.a.s.\ $\DD\left(n,p\right)$
has a Hamilton cycle, which is best possible. In fact he proved a ``hitting time'' version,
as follows. Let
\[
m_{*}=\min\left\{ m:\delta^{+}\left(D_{m}\right),\delta^{-}\left(D_{m}\right)\ge1\right\}
\]
be the first point in time that in the random digraph process, all the in-/out-degrees are at least one.
Clearly, if $m<m_{*}$ then $D_{m}$ cannot have a directed Hamilton
cycle; Frieze proved that $D_{m_{*}}$ a.a.s.\ has a directed Hamilton
cycle.

Regarding the number of Hamilton cycles in random digraphs, Janson's methods (see \cite[Theorems~10~and~11]{Jan94})
give precise control over this number in $\DD\left(n,p\right)$
for $p\gg n^{-1/2}$ (this notation means $p=\omega\left(n^{-1/2}\right)$),
but the case of sparse random digraphs seems more challenging. The
previous best result was due to Ferber and Long \cite{FL} (improving an earlier result \cite{FKL}),
that if $p\gg\log n\log \log n/n$ then $\DD\left(n,p\right)$
a.a.s.\ has $n!\left(p\left(1+o\left(1\right)\right)\right)^{n}$ Hamilton cycles.
Here we obtain the optimal result that the same estimate holds as soon
as the hitting time for existence is reached. Our proof is relatively short,
using Frieze's machinery for proving existence of Hamilton cycles, permanent estimates, some elementary facts about random permutations and simple double-counting arguments.

\begin{thm}
\label{conj:ham}If $m=n\log n+\omega\left(n\right)$ then $D_{m}$
a.a.s.\ has $n!\,\left(p\left(1+o\left(1\right)\right)\right)^{n}$
directed Hamilton cycles, where $p={m}/N=m/\left(n(n-1)\right)$.
\end{thm}

\begin{thm}
\label{conj:ham-stopping}$D_{m_{*}}$ a.a.s.\ has $n!\,\left(p\left(1+o\left(1\right)\right)\right)^{n}$
directed Hamilton cycles, where $p=\log n/n$.
\end{thm}

\subsection{Notation}

We use standard graph-theoretic notation throughout. Directed edges are ordered pairs of vertices, and the set of edges of a digraph $D$ is denoted $E(D)$. The minimum in-degree of $D$ is denoted $\delta^-(D)$, and the minimum out-degree is denoted $\delta^+(D)$.

We also use standard asymptotic notation throughout, as follows.
For functions $f=f(n)$ and $g=g(n)$, we write $f=O(g)$ to mean that there
is a constant $C$ such that $|f|\le C|g|$, $f=\Omega(g)$
to mean that there is a constant $c>0$ such that $f\ge c|g|$,
$f=\Theta(g)$ to mean that $f=O(g)$ and $f=\Omega(g)$, and 
$f=o(g)$, $g=\omega(f)$ or $f\ll g$ to mean that $f/g\to0$ as $n \to \infty$. By ``asymptotically almost surely'', or ``a.a.s.'', we mean that the probability of an event is $1-o\left(1\right)$.

For a positive integer $n$, we write $[n]$ for the set $\left\{ 1,2,\dots,n\right\} $, and write $[n]^2$ for the set of ordered pairs of such integers. For a real number $x$, the floor and ceiling functions are denoted
$\floor x=\max\{i\in\mathbb{Z}:i\le x\}$ and $\ceil x=\min\{i\in\mathbb{Z}:i\ge x\}$. Finally, all logs are base
$e$.

\section{Proof outline and ingredients}\label{sec:outline}

The upper bounds in \cref{conj:ham,conj:ham-stopping} will follow from a straightforward application of Markov's inequality, so the important contribution of this paper is to establish the lower bounds. The essential ingredient of their proofs
is the machinery of Frieze \cite{Fri88} which he developed to show
existence of a Hamilton cycle in $D_{m_{*}}$. A \emph{1-factor} in a digraph $D$ is a spanning subgraph with all
in-/out-degrees equal to 1; equivalently it is a union of directed
cycles spanning the vertex set of $D$. Very roughly speaking, Frieze's approach was to expose the edges of $D_{m_*}$ in two ``phases''. In the edges of the first phase, he proved the existence of a special kind of 1-factor, and then the edges of the second phase were used to transform this 1-factor into a Hamilton cycle. Our general approach is to use permanent estimates to show the existence of \emph{many} 1-factors in the first phase, and then use Frieze's tools and a random permutation trick to show that most of these can be completed to a Hamilton cycle. This approach only accesses a small fraction of all the Hamilton cycles in $D_{m_*}$, so we will then use some double-counting arguments to finish the proof.

Now we give a more precise outline of our proof approach. In order to state Frieze's machinery as a lemma, we first introduce some definitions that will
continue to be used throughout the paper. First, we consider some
alternative models of random digraphs with loops. Let $\DD'\left(n,p\right)$
be the random digraph where each of the $n^{2}$ possible directed
edges (including loops) is present with probability $p$ independently.
Let $e_{1}',\dots,e_{n^{2}}'$ be a random ordering of the pairs in $[n]^{2}$,
let $D_{m}'=\left\{ e_{1}',\dots,e_{m}'\right\} $, and let $m_{*}'=\min\left\{ m:\delta^{+}\left(D_{m}'\right),\delta^{-}\left(D_{m}'\right)\ge1\right\} $.
Couple $\left(e_{m}\right)_{m}$ and $\left(e_{m}'\right)_{m}$ in
such a way that for every $m\le N$, every non-loop edge of $D_{m}'$
is also in $D_{m}$.

Next, as in Frieze's paper define
\begin{equation}
m_{0}=\floor{n\log n-n\log\log\log n},\quad m_{1}=\floor{n\log n+n\log\log\log n},\quad m_{3}=\ceil{\frac23n\log n}.\label{eq:m}
\end{equation}
As in Section 4 of Frieze's paper, it is straightforward to show that a.a.s.\ $m_{0}\le m_{*},m_{*}'\le m_{1}$ (we remark that our random variable $m_*$ is called $m^*$ in Frieze's paper). The following
lemma easily follows from Frieze's methods. In \cref{app:frieze},
for the convenience of the reader we will explain how this lemma can
be directly deduced from specific parts of Frieze's paper and parts
of a paper of Lee, Sudakov and Vilenchik \cite{LSV12} which gives
a different presentation of Frieze's approach. Say a 1-factor in a digraph is \emph{good} if it has $O\left(\log\log n\right)$
loops, and $O\left(\log n\right)$ cycles in total.
\begin{lem}
Let $\mathrm{LARGE}$ be the set of vertices whose in-/out-degrees are at least
$\ceil{3\log n/\log\log n}$ in $D_{m_{3}}'$. Let $D_{*}'$ consist
of $D_{m_{3}}'$, in addition to the set of edges of $D_{m_{*}'}'$
that involve a vertex not in $\mathrm{LARGE}$. Then, the following
hold.

\label{lem:frieze}
\begin{enumerate}
\item [(1)]$D_{*}'$ a.a.s.\ has a 1-factor.
\item [(2)]$D_{*}'$ a.a.s.\ satisfies the following properties.
\begin{enumerate}[itemsep=0.5pt]
\item $\left|\mathrm{LARGE}\right|=n-O\left(\sqrt{n}\right)$;
\item there are no two points in $[n]\backslash\mathrm{LARGE}$ within distance
$10$ in $D_{*}'$;
\item every cycle in $D_{*}'$ which has length at most 3 is contained in
$\mathrm{LARGE}$;
\item every vertex in $D_*'$ has in-/out-degree at most $\log^{2}n$.
\end{enumerate}
\item [(3)]Conditioning on $D_{*}'$ satisfying (2), for any good 1-factor
$M\subseteq D_{*}'$, a.a.s.\ $D_{m_{*}}$ contains a directed Hamilton
cycle sharing $n-O\left(\log^{2}n\right)$ edges with $M$.
\end{enumerate}
\end{lem}

To prove \cref{conj:ham-stopping,conj:ham}, it will essentially suffice to estimate the number of good 1-factors in $D_{*}'$. Indeed, by Markov's inequality and \cref{lem:frieze}, a.a.s.\ almost all of these can be completed to a Hamilton cycle in $D_{m_*}$, giving an a.a.s.\ estimate for the number of Hamilton cycles in $D_{m_*}$ that are almost completely contained in $D_{*}'$. Simple double-counting arguments can then be used to leverage this estimate into an appropriate a.a.s.\ lower bound on the total number of Hamilton cycles in $D_m$ or $D_{m_*}$.

Now, in order to estimate the number of good 1-factors in $D_{*}'$, we will initially ignore the goodness requirement and give an a.a.s.\ lower bound on the total number of 1-factors in $D_{*}'$ (in \cref{lem:count-1-factors} in the next section). This will be accomplished with a greedy matching argument combined with the following lemma of Glebov and Krivelevich \cite[Lemma~4]{GK13}. This lemma conveniently
summarises the application of the Ore-Ryser theorem \cite{Ore62}
and Egorychev-Falikman theorem \cite{Ego81,Fal81} to give a lower
bound for the number of 1-factors in a pseudorandom almost-regular
digraph. (We remark that the statement of \cite[Lemma~4]{GK13} is for oriented graphs, but the proof applies equally well in the setting of arbitrary directed graphs.)
\begin{lem}
\label{lem:GK}Let $D$ be a directed graph on $[n]$ (with loops
allowed), and consider some $r=r\left(n\right)\gg\log\log n$. Suppose
that all in-degrees and out-degrees of $G$ lie in the range
\[
\left(1\pm\frac{4}{\log\log n}\right)r
\]
and suppose that for any $X_{1},X_{2}\subset [n]$
with $\left|X_{1}\right|,\left|X_{2}\right|\le\frac{3}{5}n$, the
number of edges from $X_{1}$ to $X_{2}$ is at most
\[
\frac{4r}{5}\sqrt{\left|X_{1}\right|\left|X_{2}\right|}.
\]
Then $G$ contains at least
\[
\left(\frac{r-o\left(r\right)}{e}\right)^{n}
\]
1-factors.
\end{lem}

The final step is to show that a large fraction of the 1-factors in $D_*'$ are good. For this, observe that there is a natural correspondence between directed graphs (with loops) on the vertex set $[n]$, and bipartite graphs with vertex set $[n]\sqcup[n]$. Indeed, given a directed graph $D$ on the vertex set $[n]$, consider the bipartite graph with parts $V_1,V_2$ which are disjoint copies of $[n]$, and with an edge between $x\in V_1$ and $y\in V_2$ for each $(x,y)\in E(D)$. For the rest of the paper, we will sometimes view directed graphs with loops as bipartite graphs, wherever it is convenient. Now, the distribution of $D_*'$ (as a bipartite graph) is invariant under permutations of $V_2$, so we can realise the distribution of $D_*'$ via a two-phase procedure that first generates a random instance of $D_*'$ then randomly permutes $V_2$. Using Markov's inequality, it will then suffice to show that for any particular directed 1-factor in $D_*'$, the random permutation will a.a.s.\ cause it to become good. This will be accomplished with well-known results about the cycle structure of a random permutation.

We end this section with two different versions of the Chernoff bound, which will be useful in the proofs. The
first follows from \cite[Corollary~A.1.10]{AS} and the second appears
in \cite[Corollary~2.3]{JLR00}.
\begin{lem}
\label{lem:Chernoff}Suppose $X\in\operatorname{Bin}\left(n,p\right)$.
\begin{enumerate}
\item [(1)]If $a>4\E X$ then $\Pr\left(X\ge a\right)<\exp\left(-\left(a-1\right)\log\left(a/\E X\right)\right).$
\item [(2)]If $\varepsilon\le3/2$ then $\Pr\left(\left|X-\E X\right|\ge\varepsilon\E X\right)\le2\exp\left(-\left(\varepsilon^{2}/3\right)\E X\right).$
\end{enumerate}
\end{lem}

\section{Proofs of \texorpdfstring{\cref{conj:ham,conj:ham-stopping}}{Theorems \ref{conj:ham} and \ref{conj:ham-stopping}}}

First, we deal with the upper bounds. For $m\ge\left(\log n+\omega\left(1\right)\right)/n$,
in the binomial random digraph $\DD\left(n,m/N\right)$, the expected
number $\E X$ of Hamilton cycles is $\left(n-1\right)!\left(m/N\right)^{n}$,
so by Markov's inequality and the fact that $n=\left(1+o\left(1\right)\right)^{n}$,
the probability there are more than $n^2\E X=n!\left(m/N\right)^{n}\left(1+o\left(1\right)\right)^{n}$
Hamilton cycles is $o\left(1/n\right)$. Now, the Pittel inequality (see \cite[p.~17]{JLR00}) says that if an event holds with probability $1-o\left(1/n\right)$ in the binomial model $\DD\left(n,m/N\right)$ then it holds a.a.s.\ in the uniform model $\DD\left(n,m\right)$, and applying this to $D_m$ gives the required a.a.s.\ upper bound in \cref{conj:ham}. For the
upper bound in \cref{conj:ham-stopping} we can simply recall the definition of $m_1$ in \cref{eq:m}, observe
that a.a.s.\ $D_{m_{*}}\subseteq D_{m_{1}}$ and $m_{*}=\left(1-o\left(1\right)\right)m_1$, and apply the upper bound of \cref{conj:ham} with $m=m_1$.

Now we prove the lower bounds. As outlined in \cref{sec:outline}, we first prove the following lemma.
\begin{lem}
\label{lem:count-1-factors}A.a.s.\ $D_{*}'$ contains at least
\[
\left(\left(1-o\left(1\right)\right)\frac{2\log n}{3e}\right)^{n}
\]
1-factors.
\end{lem}

The following lemma proves some simple pseudorandomness properties
of $D_{m_{3}}'$, which we will use to apply \cref{lem:GK} to prove
\cref{lem:count-1-factors}.
\begin{lem}
\label{lem:pseudorandom}For vertex subsets $X_{1},X_{2}$ of a digraph,
let $e\left(X_{1},X_{2}\right)$ be the number of edges from $X_{1}$
to $X_{2}$. The following properties hold a.a.s.\ in $D_{m_{3}}'$.
\begin{enumerate}
\item [(1)]For any $X_{1}\subseteq [n]$ and $X_{2}\subseteq [n]$,
if $\left|X_{1}\right|\left|X_{2}\right|\gg n^{2}/\log n$ then
\[
\left|e\left(X_{1},X_{2}\right)-\left|X_{1}\right|\left|X_{2}\right|\frac{m_{3}}{n^{2}}\right|\le4\sqrt{\left|X_{1}\right|\left|X_{2}\right|\frac{m_{3}}{n}}.
\]
\item [(2)]For any $X_{1}\subset [n]$ and $X_{2}\subset [n]$ with
$\left|X_{1}\right|,\left|X_{2}\right|\le\frac{3}{5}n$, we have
\[
e\left(X_{1},X_{2}\right)\le\frac{4m_{3}}{5n}\sqrt{\left|X_{1}\right|\left|X_{2}\right|}.
\]
\end{enumerate}
\end{lem}

\begin{proof}
Using the Pittel inequality (see \cite[p.~17]{JLR00}), it suffices
to instead prove that these properties hold with probability $1-o\left(1/n\right)$
in the binomial random digraph $\DD'\left(n,m_{3}/n^{2}\right)$.
Let $x_{1}=\left|X_{1}\right|$ and $x_{2}=\left|X_{2}\right|$; for
both properties, without loss of generality it suffices to consider the case where $x_{2}\ge x_{1}$.

For the first property, note that
\[
\E\left[e\left(X_{1},X_{2}\right)\right]=x_{1}x_{2}\frac{m_{3}}{n^{2}}\gg n,
\]
implying also that $\E e\left(X_{1},X_{2}\right)\gg4\sqrt{x_{1}x_{2}m_{3}/n}$.
So by part (2) of \cref{lem:Chernoff}, we have
\[
\Pr\left(\left|e\left(X_{1},X_{2}\right)-x_{1}x_{2}\frac{m_{3}}{n^{2}}\right|\ge4\sqrt{x_{1}x_{2}\frac{m_{3}}{n}}\right)\le2\exp\left(-3n\right)\ll 2^{-2n}/n.
\]
As there are at most $2^{2n}$ viable choices of $X_{1}$ and $X_{2}$,
we obtain the first property by applying the union bound.

For the second property, note that if $x_{1},x_{2}\le\frac{3}{5}n$
then
\[
\frac{4m_{3}}{5n}\sqrt{x_{1}x_{2}}\ge x_{1}x_{2}\frac{m_{3}}{n^{2}}+4\sqrt{x_{1}x_{2}\frac{m_{3}}{n}},
\]
so if $x_{2}\ge x_{1}\ge n/\log\log n$ (implying that $x_{1}x_{2}\gg n^2/\log n$)
then the second property follows from the first. It remains to consider
the case where $x_{1}<n/\log\log n$. Note that the in-/out-degree
of each vertex is binomially distributed with mean $m_{3}/n\le\log n$,
so by part (1) of \cref{lem:Chernoff} and the union bound, with probability
$1-o\left(1/n\right)$ each in-/out-degree is at most $5\log n$.
This means that if say $x_{2}\ge100x_{1}$ then the number of edges
from $X_{1}$ to $X_{2}$ is at most $\left(5\log n\right)x_{1}<\left(4/5\right)\left(m_{3}/n\right)\sqrt{x_{1}x_{2}}$
and the second property is satisfied. So we only need to deal with
the case where $x_{1}\le x_{2}\le100x_{1}<100n/\log\log n$. Again
using part (1) of \cref{lem:Chernoff}, the fact that $\sqrt{x_1x_2}=\Theta(x_2)$, and the evaluation of $\E[e(X_1,X_2)]$ at the beginning of the proof, we have
\begin{align*}
\Pr\left(e\left(X_{1},X_{2}\right)>\frac{4m_{3}}{5n}\sqrt{x_{1}x_{2}}\right) & =\exp\left(-\omega\left(x_{2}\log n\right)\right).
\end{align*}
Noting that $\binom{n}{x_{1}}\binom{n}{x_{2}}\le\left(ne/x_{1}\right)^{x_{1}}\left(ne/x_{2}\right)^{x_{2}}=\exp\left(O\left(x_{2}\log\left(n/x_{2}\right)\right)\right)$,
the desired result follows using the union bound.
\end{proof}
Now we are ready to prove \cref{lem:count-1-factors}.
\begin{proof}[Proof of \cref{lem:count-1-factors}]
Recall that digraphs with loops on the vertex set $[n]$ can be
equivalently viewed as bipartite graphs with bipartition $[n]\sqcup[n]=V_{1}\cup V_{2}$, where an edge $(x,y)\in V_1\times V_2$ appears if and only if $(x,y)\in E(D)$.
A 1-factor in such a directed graph corresponds to a perfect matching in the corresponding bipartite graph. Under this equivalence, by part (1)
of \cref{lem:frieze} we know that $D_{*}'$ a.a.s.\ contains a perfect
matching, and by definition it contains $D_{m_{3}}'$, so it suffices
to prove that a.a.s.\ for any perfect matching $M$, the bipartite graph $G=D_{m_{3}}'\cup M$
has the desired number of perfect matchings. So, assume that the properties in \cref{lem:pseudorandom} hold, and consider a perfect matching $M$. First we set aside a
small subgraph of $M$ containing the vertices with irregular degree.
We do this greedily: as long as there is a ``bad'' vertex $v$ with
degree less than $\left(1-4/\log\log n\right)m_{3}/n$ or greater
than $\left(1+4/\log\log n\right)m_{3}/n$, take the edge of $M$
containing $v$ and remove its vertices from $G$. We claim that this
process deletes fewer than $4n/(\log\log n)^2$ pairs of vertices
before terminating. Indeed, suppose that at some stage of the process
$4n/(\log\log n)^2$ pairs have been deleted. Then at least $n/(\log\log n)^2$
of these pairs involve bad vertices that are all in the same part,
and are all bad in the same way (they all either have too-high degree
or too-low degree). Let $X_{1}$ be a subset of $n/(\log\log n)^2$ of these
bad vertices, assuming without loss of generality that they are all
in $V_1$. Also, let $Y_{2}$ be the set of $4n/(\log\log n)^2$
vertices of $V_{2}$ that have been deleted. Now, we have either
\[
e\left(X_{1},V_{2}\right)>\left(1+\frac{4}{\log\log n}\right)\frac{m_{3}}{n}\cdot\frac{n}{(\log\log n)^2}
\]
or
\[
e\left(X_{1},V_{2}\backslash Y_{2}\right)<\left(1-\frac{4}{\log\log n}\right)\frac{m_{3}}{n}\cdot\frac{n}{(\log\log n)^2}.
\]
In the first case, we have
\[
e\left(X_{1},V_{2}\right)-\left|X_{1}\right|\left|V_{2}\right|\frac{m_{3}}{n^{2}}>\frac{4m_{3}}{(\log\log n)^3}>4\sqrt{\left|X_{1}\right|\left|V_{2}\right|\frac{m_{3}}{n}}
\]
and in the second case we have
\[
\left|X_{1}\right|\left|V_{2}\backslash Y_{2}\right|\frac{m_{3}}{n^{2}}-e\left(X_{1},V_{2}\backslash Y_{2}\right)>\frac{4m_{3}}{(\log\log n)^3}\left(1-\frac{1}{\log\log n}\right)>4\sqrt{\left|X_{1}\right|\left|V_{2}\backslash Y_{2}\right|\frac{m_{3}}{n}},
\]
both of which contradict property (1) of \cref{lem:pseudorandom}.
So, after at most $4n/(\log\log n)^2$ deletions, we obtain a bipartite graph (equivalently, directed graph)
$D$ with all degrees in the range
\[
\frac{m_3}{n}\left(1\pm\frac{4}{\log\log n}\right).
\]
Applying \cref{lem:GK} (using the second property of \cref{lem:pseudorandom}),
 $D$ has
\[
\left(\left(1-o\left(1\right)\right)\frac{m_3}{en}\right)^{n-O\left(n/(\log\log n)^2\right)}=\left(\left(1-o\left(1\right)\right)\frac{2\log n}{3e}\right)^{n}
\]
perfect matchings, which can each be combined with the deleted edges
of $M$ to give the desired number of perfect matchings in $G$.
\end{proof}
Now we prove \cref{conj:ham-stopping}.
\begin{proof}[Proof of \cref{conj:ham-stopping}]
Let $\sigma$ be a uniformly random permutation of $[n]$. For
a directed edge $e=\left(v,w\right)$, let $\sigma\left(e\right)=\left(v,\sigma\left(w\right)\right)$,
and for a digraph $D$ let $\sigma\left(D\right)=\left\{ \sigma\left(e\right):e\in D\right\} $.
Conditioning on $D_{*}'$, for any directed 1-factor $M$ in $D_{*}'$, note
that $\sigma\left(M\right)$ corresponds to a uniformly  random permutation
of $[n]$, so a.a.s.\ has fewer than $\log\log n$ loops (the
expected number of such is exactly 1), and has fewer than $2\log n$
cycles (see for example \cite[Theorem~14.28]{Bol01}). By Markov's
inequality, a.a.s.\ at most a $o\left(1\right)$-fraction of the 1-factors
in $\sigma\left(D_{*}'\right)$ have more than $2\log n$ cycles or
more than $\log\log n$ loops. Note that $\sigma\left(D_{*}'\right)$
actually has the same distribution as $D_{*}'$ (because $\sigma\left(D_{*}'\right)$
can be obtained with the same definition as $D_{*}'$, using the sequence
of edges $\sigma\left(e_{1}'\right),\dots,\sigma\left(e_{n^{2}}'\right)$
in place of $e_{1}',\dots,e_{n^{2}}'$). So, we have proved that a.a.s.\ $D_{*}'$ contains at least
\[
\left(\left(1-o\left(1\right)\right)\frac{2\log n}{3e}\right)^{n}
\]
good 1-factors. Condition on such an outcome of $D_{*}'$ also satisfying
part (2) of \cref{lem:frieze}. Note that every 1-factor in $D_*'$ has all but $2(n-|\mathrm{LARGE}|)=O(\sqrt n)$ of its directed edges between vertices in $\mathrm{LARGE}$, and all other edges are also present in $D_{m_3}'$. Also, by part (3) of \cref{lem:frieze}
and Markov's inequality, a.a.s.\ at most an $o\left(1\right)$-fraction of
the good 1-factors in $D_*'$ cannot be transformed into a Hamilton cycle in $D_{m_*}$ by modifying fewer than $O(\log^2 n)$ of their edges. Since $D_*'$ has $m_3+O(\sqrt n \log^2 n)=O(n\log n)$ edges, for any given Hamilton cycle in $D_{m_*}$, the number of 1-factors it could have been transformed from is at most $$\sum_{i\le O(\log^2n)}\binom n{i}\binom {O(n \log n)}{i}=(1+o(1))^n.$$ Therefore, in $D_{m_*}$, there are a.a.s.\ at least
\[
\left(\left(1-o\left(1\right)\right)\frac{2\log n}{3e}\right)^{n}
\]
Hamilton cycles which have all but at most $O\left(\sqrt{n}+\log^2 n\right)=O\left(\sqrt{n}\right)$
of their edges in $D_{m_{3}}$. (Note that every non-loop edge of $D_{m_3}'$ is also in $D_{m_3}$). We say such Hamilton cycles are \emph{almost-contained} in $D_{m_{3}}$.

Let $I$ be a uniformly random subset of $[m_{0}]$ of size
$m_{3}$, and let $D_{I}=\left\{ e_{i}:i\in I\right\} $. Let $X$
be the number of Hamilton cycles in $D_{m_{*}}$ and let $X_{I}$
be the number of Hamilton cycles in $D_{m_{*}}$ that are almost-contained
in $D_{I}$. Conditioning on $D_{m_{*}}$, and considering any Hamilton cycle
$H\in D_{m_{*}}$, for any $i$ there are $\binom{n}{i}\binom{m_{0}-n}{m_{3}-\left(n-i\right)}$ possibilities for $I$ containing all but $i$ edges of $H$. We can then compute that the probability $H$ is almost-contained in $D_{I}$ is
\begin{align*}
&\frac{\sum_{i\le O(\sqrt{n})}\binom{n}{i}\binom{m_{0}-n}{m_{3}-\left(n-i\right)}}{\binom{m_{0}}{m_{3}}}\\
 &\quad =\sum_{i\le O(\sqrt{n})}\binom{n}{i}\;(m_0-m_3)(m_0-m_3-1)\dots(m_0-m_3-i+1)\frac{m_3(m_3-1)\dots(m_3-n+i+1)}{m_0(m_0-1)\dots(m_0-n+1)}\\
&\quad =(n\log n)^{O(\sqrt n)}\;\frac{m_3(m_3-1)\dots(m_3-n+1)}{m_0(m_0-1)\dots(m_0-n+1)}=\left(\frac{2}{3}+o\left(1\right)\right)^{n}.
\end{align*}
So,
$
\E\left[X_{I}\,\middle|\, D_{m_{*}}\right]=\left(2/3+o\left(1\right)\right)^{n}X.
$
Then, Markov's inequality says that a.a.s.
\[
X_{I}\le n\left(\frac{2}{3}+o\left(1\right)\right)^{n}X=\left(\frac{2}{3}+o\left(1\right)\right)^{n}X.
\]
On the other hand, conditioning on the event $m_{*}\ge m_{0}$ (which
holds a.a.s.), each $X_{I}$ has the same distribution as $X_{\left[m_{3}\right]}$,
so a.a.s.
\[
X_{I}\ge\left(\left(1-o\left(1\right)\right)\frac{2\log n}{3e}\right)^{n}.
\]
The desired result that a.a.s.\ $X\ge ((1-o(1))\log n/e)^n=n!((1+o(1))\log n/n)^n$ follows.
\end{proof}
Finally, we deduce \cref{conj:ham} from \cref{conj:ham-stopping}.
\begin{proof}[Proof of \cref{conj:ham}]
Choose $m_{1}'\le m$ such that $m_{1}'=n\log n+\omega\left(n\right)$
and $m_{1}'\sim n\log n$. We a.a.s.\ have $m_{*}\le m_{1}'$ so a.a.s.\ $D_{m_{1}'}$ has at least $\left(\left(1+o\left(1\right)\right)\log n/e\right)^{n}$
Hamilton cycles, by \cref{conj:ham-stopping}. Let $I$ be a uniformly
random subset of $[m]$ of size $m_{1}'$, and let $D_{I}=\left\{ e_{i}:i\in I\right\} $.
Let $X$ be the number of Hamilton cycles in $D_{m}$ and let $X_{I}$
be the number of Hamilton cycles in $D_{I}$. Conditioning on $D_{m}$,
for any Hamilton cycle $H$ in $D_{m}$ we have
\[
\Pr\left(H\subseteq D_{I}\right)=\frac{\binom{m-n}{m_{1}'-n}}{\binom{m}{m_{1}'}}=\left(\left(1+o\left(1\right)\right)\frac{m_{1}'}{m}\right)^{n}=\left(\left(1+o\left(1\right)\right)\frac{n\log n}{m}\right)^{n},
\]
so $\E\left[X_{I}\,\middle|\, D_{m_{*}}\right]=\left(\left(1+o\left(1\right)\right)n\log n/m\right)^{n}X$
and by Markov's inequality, a.a.s.
\[
X_{I}\le\left(\left(1+o\left(1\right)\right)\frac{n\log n}{m}\right)^{n}X.
\]
On the other hand, by \cref{conj:ham-stopping} and the fact that each $X_I$ has the same distribution as $X_{[m_1']}$ we a.a.s.\  have $X_{I}\ge\left(\left(1-o\left(1\right)\right)\log n/e\right)^{n}$,
and the desired result that a.a.s.\ $X\ge ((1-o(1))m/(ne))^n=n!((1-o(1))m/N)^n$ follows.
\end{proof}

\providecommand{\bysame}{\leavevmode\hbox to3em{\hrulefill}\thinspace}


\begin{appendices}
\crefalias{section}{appsec}

\section{\label{app:frieze}Discussion of \texorpdfstring{\cref{lem:frieze}}{Lemma~\ref{lem:frieze}}}

In this section we justify \cref{lem:frieze}. First, note that part
(1) is proved in ``phase 1'' of Frieze's paper \cite[Section~4]{Fri88}.
Specifically, he considers a digraph $E^{1}=E^{1+}\cup E^{1-}$, where
$E^{1+}\subseteq D_{m_{*}}'$ consists of the first 10 edges pointing
away from each vertex $v$ (or as many as possible if $d^{+}\left(v\right)<10$),
and then $E^{1-}$ consists of the first 10 edges pointing towards
each vertex $v$ (or as many as possible), disjoint to the edges of
$E^{1+}$. Note that $E^{1}\subseteq D_{*}'$. He then gives an algorithm
to produce a 1-factor in $E^{1}$, and shows that this algorithm a.a.s.\ succeeds. Although he states this algorithm for the loopless model,
his proof that it a.a.s.\ succeeds (in his Lemma 4.2) starts by showing
the corresponding fact for the model with loops, which is what we
need.

Next, part (2) is very routine. For each vertex $v$, a concentration
inequality for the hypergeometric distribution (for example, \cite[Theorem~2.10]{JLR00})
shows that the probability that $d^{-}\left(v\right)>\log^{2}n$ or
$d^{+}\left(v\right)>\log^{2}n$ in $D'_{m_{1}}$ is $o\left(1/n\right)$.
Since a.a.s.\ $D'_{m_{1}}\supseteq D_*'$, (d) immediately follows from the union
bound. Then, (a) appears as Lemma 5.1 in Frieze's paper and (b) appears
as Lemma 6.4. Specifically, his proof of Lemma 5.1 (which is basically
an application of the Chernoff bound) shows that the probability that
a particular vertex $v$ is not in $\mathrm{LARGE}$ is $O\left(n^{-2/3}\right)$.
Now, let $d^-(v),d^+(v)$ be the in-/out-degrees of $v$ in
$D_{m_{1}}'$, and condition on any particular outcomes of the values of $d^-(v),d^+(v)$, each at most $\log^{2}n$ (this conditioning determines
whether $v\in\mathrm{LARGE}$). In $D_{m_{1}}'$, the probability $v$ is a loop is $O\left((d^{-}\left(v\right)+d^{+}\left(v\right))/n\right)=O(\log^2n/n)$, the probability it is in a cycle of length 2 is $O(n(d^{-}\left(v\right)/n)(d^{+}\left(v\right)/n))=O(\log^4 n/n)$, and the probability it is in a cycle of length 3 is $O(n^{2}(d^{-}\left(v\right)/n)(d^{+}\left(v\right)/n)(m_{1}/n^{2}))=O\left(\log^{5}n/n\right)$. Therefore the total probability $v$ is involved in a cycle
of length at most 3 in $D_{m_{1}}'$ is $O\left(\log^{5}n/n\right)$. Recalling that $d^{-}\left(v\right),d^{+}\left(v\right)\le\log^{2}n$
with probability $1-o\left(1/n\right)$, the unconditional probability that $v$
is outside $\mathrm{LARGE}$ and is also involved in a cycle of length at
most 3 in $D_{m_{1}}'$ is $O\left(n^{-2/3}\log^{5}n/n\right)+o\left(1/n\right)=o\left(1/n\right)$.
By the union bound, a.a.s.\ there is no such vertex. Using the fact
that a.a.s.\ $D_{*}'\subseteq D_{m_{1}}'$, it follows that in $D_{*}'$
a.a.s.\ every cycle of length at most 3 is contained in $\mathrm{LARGE}$,
proving (c).

It remains to justify part (3). Let $D_{*}$ consist of $D_{m_{3}}$,
in addition to the set of edges of $D_{m_{*}}$ that involve a vertex
not in $\mathrm{LARGE}$. We are conditioning on an outcome of $D_{*}'$ satisfying
(2); additionally condition on a consistent outcome of $D_{*}$. Now consider any good 1-factor $M\subseteq D_*'$ (recall that this means $M$ has at $O(\log \log n)$ loops and $O(\log n)$ cycles in total). Note
that our good 1-factor $M\subseteq D_{*}'$ is also a subgraph of $D_{*}$,
except for its loops.

Since every loop of $M$ is in $\mathrm{LARGE}$, its vertex $v$ has degree $\Omega(\log n/\log \log n)$ and is therefore adjacent
to some vertex $v_{0}$ in another cycle $C=v_{0}v_{1}v_{2}\dots v_{\ell}v_{0}$
of $M$. By adding the ``virtual edge'' $vv_{1}$, if necessary,
we can merge the loop with $C$. Do this repeatedly (avoiding the vertices used in previous steps) until there
are no loops left, resulting in a vertex-disjoint set $E_{\mathrm{V}}$
of $O(\log\log n)$ virtual edges and a 1-factor $M_{\mathrm{V}}\subseteq E_{\mathrm{V}}\cup D_{*}$
with no loops, sharing $n-O\left(\log\log n\right)$ of its edges
with $M$. Note that new cycles of length at most 3 can only be created by merging loops and cycles of length 2, so all such cycles are contained in $\mathrm{LARGE}$.

Now, since $m_{0}-m_{3}=\Omega\left(n\log n\right)$, by (2a) and (2d) there are $\Omega\left(n\log n\right)-O(\sqrt{n}\log^{2}n)=\Omega\left(n\log n\right)$
edges of $D_{m_{*}}$ that still have not been exposed. Conditionally,
these comprise a uniformly random subset of $\Omega\left(n \log n\right)$
edges between vertices in $\mathrm{LARGE}$. Therefore it suffices to prove
the following lemma.
\begin{lem}\label{lem:app-lem}
Let $M$ be a 1-factor on the vertex set $[n]$ with no loops
and $O\left(\log n\right)$ cycles. Consider a set of vertices $L\subseteq [n]$ and consider a vertex-disjoint set $E_{\mathrm{V}}$ of $O\left(\left(\log\log n\right)^{3}\right)$``forbidden edges''.
Let $D$ be the directed graph obtained by adding a set of $\Omega\left(n\log n\right)$
uniformly random non-loop edges between vertices of $L$, to $M$. Now, suppose the following conditions are satisfied.
\begin{enumerate}[itemsep=0.5pt]
\item[(a)]{$|L|=n-O\left(\sqrt{n}\right)$;}
\item[(b)]{in $M$ there are no two vertices outside $L$
within distance 10;}
\item[(c)]{every cycle in $M$ of length at
most 3 is contained in $L$.}
\end{enumerate}
Then a.a.s.\ $D$ has a directed Hamilton cycle containing no edge
of $E_{\mathrm{V}}$ and sharing $n-O\left(\log^{2}n\right)$ of its
edges with $M$.
\end{lem}

This lemma follows from what is proved in Sections 5-6 of the arXiv
version of the paper of Lee, Sudakov and Vilenchik \cite{LSV12} (compare with Lemma 3.1 of that paper). We
outline the details. First, one reduces to the case where $L=[n]$
via a ``compression'' argument. Specifically, start with $D$ and for each vertex $v\notin L$,
suppose $v_{-},v,v_{+}$ appear in order on some cycle of $M$. Then
we can replace $v_{-},v,v_{+}$ with a new vertex $v'$ where $v'$
takes as in-neighbours the in-neighbours of $v_{-}$ and as out-neighbours
the out-neighbours of $v_{+}$. Perform this operation repeatedly
until there are no vertices left outside $L$. Since there are no
two vertices outside $L$ within distance 10, the compression operations
do not interfere with each other, and since every cycle in $M$ of
length at most 3 is contained in $L$, no cycle of $M$ becomes a
loop. Denote the resulting compressed random digraph by $D_{\mathrm c}$, and denote its 1-factor arising from $M$ by of $M_{\mathrm c}$. Note that a 1-factor in $D_{\mathrm c}$ sharing all but $z$ of its edges with $M_{\mathrm c}$ yields a 1-factor in $D$ sharing all but $z$ of its edges with $M$, and note that the distribution of $D_{\mathrm c}$ is almost exactly the same as the distribution $\mathcal L$ obtained by adding a set of $\Omega(n\log n)$ uniformly random non-loop edges to $M_{\mathrm c}$. To be precise, the distribution of the compression of the ``binomial version'' $D^{\mathrm{bin}}$ of $D$, where every possible non-loop edge between vertices of $L$ is added to $M$ with probability $p=\Omega(n\log n)/(|L|(|L|-1))$, coincides exactly with the ``binomial version'' $\mathcal L^{\mathrm{bin}}$ of $\mathcal L$, where one adds to $M_{\mathrm c}$ every possible non-loop edge with probability $p$. For monotone events such as containment of certain types of Hamilton cycles, the binomial and uniform models are asymptotically equivalent (meaning that if such an event holds a.a.s.\ in one model, it holds a.a.s.\ in the other model; see \cite[Corollary~1.16]{JLR00}). Therefore to prove that $D$ a.a.s.\ has a Hamilton cycle of the required type, it suffices to prove that a random digraph distributed as $\mathcal L$ a.a.s.\ does.

Now we can simply apply Lemma 6.5 of the non-arXiv version of
\cite{LSV12}, which we reproduce as follows.
\begin{lem}\label{lem:LSV12}
Let $M$ be a 1-factor on a vertex set $V$ of size $\left(1-o\left(1\right)\right)n$
with no loops and $O\left(\log n\right)$ cycles. Let $D$ be the
directed graph obtained by adding a set of $\Omega\left(n\log n\right)$
uniformly random edges to $M$. Then a.a.s.\ for any set $E_{\mathrm{V}}$
of at most $\left(\log\log n\right)^{3}$ vertex-disjoint edges of
$M$, the directed graph $D-E_{\mathrm{V}}$ contains a directed Hamilton
cycle $H$.
\end{lem}

The astute reader will notice that \cref{lem:LSV12} does not quite suffice to prove \cref{lem:app-lem}, because we need the additional fact that
$H$ shares $\left|V\right|-O\left(\log^{2}n\right)$ of its edges
with the original 1-factor $M$. Although this fact is not explicitly stated in \cite[Lemma~6.5]{LSV12}, it follows immediately from the proof (which appears only in the arXiv version of \cite{LSV12}, and
uses basically the same arguments as in Frieze's paper \cite[Sections~5--6]{Fri88}).
For the convenience of the reader, we outline the steps in the proof, to make it clear why this fact holds. The idea is to manipulate $M$ into $H$ in two phases.

In the first of these phases (called ``phase 2'' in Frieze's paper), we
greedily ``patch together'' most of the cycles, leaving a single
cycle of length $n-o\left(n\right)$ and a few short cycles. To do this,
we repeatedly look for pairs of edges $\left(v_{1},w_{1}\right)$
and $\left(v_{2},w_{2}\right)$ in different cycles $C_{1}$ and $C_{2}$,
with $\left(v_{1},w_{2}\right),\left(v_{2},w_{1}\right)\in G$, so
that we can replace $C_{1}$ and $C_{2}$ with a ``merged'' cycle. In order to prove that this succeeds, we partition half of the random edges in $D$ into subsets of carefully chosen sizes, so we have independent batches of random edges for each of the merging steps (the remaining half of the random edges will be used for the next phase).

The details of the second phase (called ``phase 3'' in Frieze's paper) are rather technical, but the idea is again to iteratively merge the remaining short cycles
into the single long cycle. At each stage of this process, we have a long cycle $C_1$ and we aim to merge it with a particular short cycle $C_{i}$. To do this, we first use the random edges of $D$ to a.a.s.\ find an edge between $C_{i}$ and $C_{1}$, which allows
us to ``unravel'' the cycles into a long path spanning the vertex
set of $C_{i}\cup C_{1}$. Then we repeatedly perform ``rotations''
to our path, whereby we transform our long path into a different long path
on the same vertex set. Specifically, for a directed path $P=v_0\dots v_\ell$, if for some $1\le i<j$ the edges $v_iv_j$ and $v_\ell v_{i+1}$ are present, then we can transform $P$ into the path $v_0\dots v_i v_j \dots v_\ell v_{i+1}\dots v_{j-1}$. Considering sequences of at most $O\left(\log n\right)$ such rotations yields enough different paths with different endpoints that a.a.s.\ one of these paths has an endpoint with an edge to $v_0$, meaning this path can be closed into a cycle. This process of transforming a path into a cycle on the same vertex set is encapsulated in Lemma~6.6 of the arXiv version of~\cite{LSV12}.

After merging all the cycles in this way to obtain a Hamilton cycle, we can then perform further rotations to eliminate
any remaining edges of $E_{\mathrm{V}}$. To elaborate, we iteratively do the following. Remove an edge of $E_{\mathrm{V}}$ to obtain a Hamilton path, then use a sequence of rotations (avoiding adding new edges of $E_{\mathrm{V}}$) as in the previous paragraph to transform this path into a Hamilton cycle with fewer edges of $E_{\mathrm{V}}$.

In the first phase, $O(\log n)$ new edges are introduced into our 1-factor (two for each merge). In the second phase, we need to perform $O(\log n)$ sequences of rotations to merge the remaining cycles, and $O((\log\log n)^3)$ sequences of rotations to eliminate edges of $E_{\mathrm{V}}$. Each sequence of rotations involves $O(\log n)$ new edges (two for each individual rotation, and an additional edge to close a path into a cycle). In total, only $O\left(\log n+\log^{2}n+\log n\left(\log\log n\right)^{3}\right)=O\left(\log^{2}n\right)$
edges are changed.

\end{appendices}

\end{document}